\documentclass[11pt]{amsart}
\usepackage{amssymb,amsmath,mathtools}
	\mathtoolsset{centercolon}
\usepackage[utf8]{inputenc}
\usepackage[shortlabels]{enumitem}
	\setenumerate[0]{label={\rm (\roman*)}, leftmargin=*}
\usepackage{geometry}
\usepackage{hyperref}\hypersetup{bookmarksopen=true}
\usepackage[capitalize]{cleveref}

\numberwithin{equation}{section}

\newtheorem{theorem}[equation]{Theorem}
\newtheorem{proposition}[equation]{Proposition}
\newtheorem{lemma}[equation]{Lemma}

\theoremstyle{definition}
\newtheorem{definition}[equation]{Definition}

\theoremstyle{remark}

\newcommand{\g}{\mathfrak g}
\newcommand{\cg}{\overline{\g}}
\newcommand{\C}{\mathbb C}
\newcommand{\R}{\mathbb R}

\newcommand{\id}{\mathrm{id}}
\newcommand{\dash}{\nobreakdash-\hspace{0pt}}

\title[A 9-dimensional nilpotent Lie algebra]{A 9-dimensional nilpotent Lie algebra isomorphic to its complex conjugate admitting no real form}
\date{October 9, 2026}
\author{Tom De Medts}
\address{Tom De Medts \\ Ghent University \\ Department of Mathematics, Computer Science and Statistics \\ Krijgslaan 299, S9 \\ B-9000 Gent \\ Belgium}
\email{tom.demedts@ugent.be}

\begin{document}

\begin{abstract}
	We give an explicit example of a $9$-dimensional $5$-step nilpotent complex Lie algebra $\g$, generated by two elements, that is
	isomorphic to its complex conjugate but has no real form.
	In contrast to the earlier $10$-dimensional examples, our proof is completely elementary and very short.
\end{abstract}

\maketitle

\vspace*{-3.8ex}

\section{Introduction}

There is a subtle connection between complex Lie algebras that are isomorphic to their own complex conjugate, and complex Lie algebras admitting a real form.
It is clear that \emph{if} a complex Lie algebra admits a real form, \emph{then} it is isomorphic to its complex conjugate, simply because the former condition ensures the existence of a basis such that all structure constants are real.
It is tempting to think that the converse is also true; we refer to the MathOverflow discussion \cite{MO} containing two incorrect attempts for a proof. Notice, for instance, that the converse is indeed true for semisimple Lie algebras, and in \cite[p.\@~494]{Dere19}, Jonas Deré mentions that it is also true for all nilpotent complex Lie algebras up to dimension~$7$.

In \cite{Dere19}, Deré explicitly posed this question as a conjecture.
In \cite{Demarche26}, Cyril Demarche disproved the conjecture, by showing the existence of a $10$-dimensional $2$-step nilpotent counterexample; however, his proof is non-constructive.
Following his ideas, Mikhail Borovoi, Willem de Graaf and Robert Guralnick have recently provided an \emph{explicit} example of such a $10$-dimensional $2$-step nilpotent Lie algebra; see \cite{BdGG26}.
(They also show how to construct counterexamples of any dimension $d \geq 10$ from this.)
The construction is rather involved, and the proof relies on computer verifications in GAP and Magma.
By \cite[Remark 1.3]{BdGG26}, dimension $10$ is the smallest possible for a $2$-step nilpotent example.

\smallskip

We provide an example that is $9$-dimensional, $5$-step nilpotent, and generated by only $2$ elements.
Notice that the \emph{free} $5$-step nilpotent Lie algebra on $2$ generators has dimension~$14$, with homogeneous components of dimensions $2$, $1$, $2$, $3$ and $6$. Our example is (of course) a quotient of this algebra, and we have homogeneous components of dimensions $2$, $1$, $2$, $2$ and~$2$.
We do not know whether there exists a nilpotent $8$\dash dimensional complex Lie algebra isomorphic to its complex conjugate admitting no real form, although experiments with $2$-generated and $3$-generated nilpotent Lie algebras seem to suggest that it is unlikely.

\smallskip

In contrast to the earlier examples, our proof is completely elementary, very short, and can be verified directly by hand.

\section{Setup}

Let $\g$ be a finite-dimensional Lie algebra over $\C$ with basis $\{ b_1,\dots,b_n \}$ and structure constants $r_{ij}^k$, i.e., we have
\[ [b_i, b_j] = \sum_{k=1}^n r_{ij}^k b_k \]
for all $i,j$.
We then define the \emph{complex conjugate} $\cg$ of $\g$ to be the Lie algebra over $\C$ with basis $\{ b'_1,\dots, b'_n \}$, with Lie bracket
\begin{equation}\label{eq:conj}
	[b'_i, b'_j] = \sum_{k=1}^n \overline{r_{ij}^k} b'_k .
\end{equation}
This definition depends on the choice of the basis $\{ b_1,\dots,b_n \}$, but different choices give isomorphic algebras, so that ``$\g \cong \cg$'' is a well-defined condition.

We say that a complex Lie algebra $\g$ has a \emph{real form} $\g_0$ if $\g_0$ is a Lie algebra over $\R$ such that $\g \cong \g_0 \otimes_\R \C$, or in other words, if $\g$ admits a basis for which all structure constants $r_{ij}^k$ are real.

If $\g$ has a real form, then clearly $\g \cong \cg$.
The converse is not true, and our goal is precisely to give an elementary example of this phenomenon.

\begin{definition}
	Let $\g$ be a Lie algebra with a fixed basis $\{ b_1,\dots,b_n \}$ and let $\cg$ be its complex conjugate with basis $\{ b'_1,\dots,b'_n \}$.
	\begin{enumerate}
		\item For each $v = \sum_j \lambda_j b_j \in \g$, we define $\overline{v} := \sum_j \overline\lambda_j b_j' \in \cg$.
			Similarly, if $w = \sum_j \lambda_j b_j' \in \cg$, then we define $\overline{w} := \sum_j \overline\lambda_j b_j \in \g$.
			In particular, $\overline b_j = b'_j$ and $\overline{b'_j} = b_j$ for all $j$.
			
			Notice that the map $\g \to \cg \colon v \mapsto \overline{v}$ is $\C$-semilinear, i.e., $\overline{\lambda v} = \overline{\lambda} \overline{v}$ for all $\lambda \in \C$ and all $v \in \g$; moreover,
			\[ [\overline v, \overline w] = \overline{[v,w]} \quad \text{for all } v,w \in \g . \]
			Similar properties hold, of course, for the map $\cg \to \g \colon w \mapsto \overline{w}$.
		\item If $\psi \colon \g \to \cg$ is a $\C$-linear Lie algebra isomorphism, then we define $\overline\psi \colon \cg \to \g$ as the isomorphism given by conjugating the coordinates of $\psi$ with respect to our choice of bases. Equivalently, $\overline\psi \colon \cg \to \g$ is defined by
			\[ \overline\psi(\overline v) := \overline{\psi(v)} \quad \text{for all } v \in \g . \]
			Clearly, $\overline\psi$ is again a $\C$-linear Lie algebra isomorphism.
	\end{enumerate}
\end{definition}

The following criterion is useful.

\begin{lemma}\label{le:realform}
	The complex Lie algebra $\g$ has a real form if and only if there is a $\C$-linear isomorphism of Lie algebras $\psi \colon \g \to \cg$ with $\overline{\psi} \circ \psi = \id$.
\end{lemma}
\begin{proof}
	If $\g$ has a real form, then we can find a new basis $\{ f_1, \dots, f_n \}$ for $\g$ with respect to which all structure constants $s_{ij}^k$ belong to $\R$.
	Then $\{ \overline f_1, \dots, \overline f_n \}$ is a basis for $\cg$, and we have
	\[ [\overline f_i, \overline f_j] = \overline{[f_i, f_j]} = \overline{\sum_k s_{ij}^k f_k} = \sum_k s_{ij}^k \overline f_k \]
	since $s_{ij}^k \in \R$. Hence the $\C$-linear map $\psi \colon \g \to \cg$ defined by $\psi(f_i) := \overline f_i$ for all $i$ is a Lie algebra isomorphism. By definition, we have $\overline{\psi}(\overline f_i) = \overline{\psi(f_i)} = \overline{\overline{f_i}} = f_i$, so indeed $\overline{\psi} \circ \psi = \id$.
	
	\medskip
	
	Conversely, assume that there is a $\C$-linear isomorphism of Lie algebras $\psi \colon \g \to \cg$ with $\overline{\psi} \circ \psi = \id$.
	Composing $\psi$ with the map $\cg \to \g \colon w \mapsto \overline{w}$, we obtain a $\C$-semilinear automorphism $\rho \colon \g \to \g$, and the condition $\overline\psi \circ \psi = \id$ ensures that $\rho$ is an involution.
	The $\R$-subspace of fixed points of $\g$ under $\rho$ is then a real form of $\g$.
	(We omit the details because we do not need this reverse implication.)
\end{proof}

\section{The example}

\begin{definition}\label{def}
Let $\g$ be the $9$-dimensional complex vector space with basis
\[ x_1,\ y_1,\quad x_2,\quad x_3,\ y_3,\quad x_4,\ y_4,\quad x_5,\ y_5 \]
and the antisymmetric bracket whose nonzero values on basis vectors are
\begin{alignat*}{4}
	&[x_1, y_1] = x_2, \\
	&[x_1, x_2] = x_3, && [y_1,x_2]=y_3, \\
	&[x_1, x_3] = x_4, && [y_1,y_3]=y_4, \\
	&[x_2, x_3] = x_5, && [x_2,y_3]=y_5, \\
	&[x_1, x_4] = x_5 + i y_5, \quad && [y_1, x_4] = -x_5, \quad && [x_1, y_4] = y_5, \quad && [y_1, y_4] = i x_5 + y_5.
\end{alignat*}
All other brackets of basis vectors are zero; in particular $[y_1,x_3]=[x_1,y_3]=0$.
\end{definition}

Now set $\g_1 := \langle x_1,y_1\rangle$, $\g_2 := \langle x_2\rangle$ and $\g_k := \langle x_k,y_k\rangle$ for $k=3,4,5$, and $\g_k=0$ for $k\geq 6$.
Then $[\g_j,\g_k]\subseteq\g_{j+k}$ for all $j,k$, i.e., the bracket is graded.

\begin{proposition}\label{pr:lie}
	Let $\g$ be the algebra from \cref{def}.
	Then:
	\begin{enumerate}
		\item $\g$ is a Lie algebra.
		\item $\g$ is generated by $x_1$ and $y_1$ and is $5$-step nilpotent.
		\item The lower central series of $\g$ is $\g^k := \bigoplus_{j=k}^5 \g_j$. In particular, $\g = \langle x_1, y_1 \rangle \oplus [\g, \g]$.
	\end{enumerate}
\end{proposition}
\begin{proof}
	To prove that $\g$ is a Lie algebra, it suffices to check the Jacobi identity
	\[ J(a,b,c)=[a,[b,c]]+[b,[c,a]]+[c,[a,b]]=0 \]
	for basis vectors $a,b,c$.
	Because the bracket is graded and $\g_k=0$ for $k\ge6$, $J(a,b,c)=0$ whenever the degrees add up to at least $6$.
	We may also assume that $a,b,c$ are pairwise distinct, because $J$ is alternating.
	Since $\dim\g_1=2$ and $\dim\g_2=1$, the only remaining triples are $(x_1,y_1,x_2)$, $(x_1,y_1,x_3)$ and $(x_1,y_1,y_3)$, and we check that
	\begin{align*}
		J(x_1,y_1,x_2) &= [x_1,y_3] - [y_1,x_3] + [x_2,x_2] = 0, \\
		J(x_1,y_1,x_3) &= [x_1,0] - [y_1,x_4] + [x_3,x_2] = 0 + x_5 - x_5 = 0, \\
		J(x_1,y_1,y_3) &= [x_1,y_4] + [y_1,0] + [y_3,x_2]=y_5+0-y_5=0.
	\end{align*}
	This proves (i).
	
	It is immediately clear from \cref{def} that $\g$ is generated by $x_1$ and $y_1$.
	By construction, $\g_j = [\g_1, \g_{j-1}]$ for all $j \geq 2$, so since $\g$ is generated by $\g_1$, we see by induction on~$k$ that the lower central series of $\g$ is $\g^k := \bigoplus_{j=k}^5 \g_j$.
	In particular, $\g$ is $5$-step nilpotent. This proves (ii) and (iii).
\end{proof}

Let $\cg$ be the complex conjugate of $\g$ as defined in \eqref{eq:conj}, with respect to the basis from \cref{def}.
We denote the corresponding basis of $\cg$ by $\{ x'_1, \dots, y'_5 \}$.
Explicitly, $\cg$ has exactly the same Lie brackets of basis vectors as in \cref{def} except for
\[ [x'_1, x'_4] = x'_5 - i y'_5 \quad \text{and} \quad [y'_1, y'_4] = -i x'_5 + y'_5 . \]

\begin{proposition}\label{pr:tau}
	Let $\tau$ be the $\C$-linear map from $\g$ to $\cg$ given on the basis by
	\[ \tau(x_2) = x'_2, \qquad \tau(x_k) = y'_k \quad \text{and} \quad \tau(y_k) = (-1)^k x'_k \qquad (k=1,3,4,5) .\]
	Then $\tau$ is an isomorphism of Lie algebras. In particular, $\g \cong \cg$.
\end{proposition}
\begin{proof}
	To show that $\tau$ is an isomorphism of Lie algebras, it suffices to check that $\tau([a,b]) = [\tau(a), \tau(b)]$ for all basis vectors $a,b$ of $\g$, where the Lie bracket on the left hand side is in~$\g$ and the Lie bracket on the right hand side is in $\cg$.
	It suffices to first go over each of the $11$ non-zero brackets appearing in \cref{def} and then to check that $[\tau(a), \tau(b)] = 0$ whenever $[a,b] = 0$; the verification is straightforward.
\end{proof}

\section{Automorphisms of $\g$ act on $\g/[\g,\g]$ by scalars}

In this section, which is the key ingredient of our proof, we will show that for any $\C$\dash linear automorphism of the Lie algebra $\g$, the induced action on $\g / [\g, \g]$ is given by multiplication by a scalar.

\smallskip

So let $\varphi \colon \g \to \g$ be an arbitrary $\C$-linear Lie algebra automorphism.
 Since $\g$ is generated by $x_1$ and $y_1$ and $\g = \langle x_1, y_1 \rangle \oplus [\g, \g]$, the automorphism $\varphi$ is completely determined by
\begin{alignat*}{2}
	\varphi(x_1) &= X + v, \qquad & X &= \alpha x_1 + \beta y_1, \quad v \in [\g, \g], \\
	\varphi(y_1) &= Y + w, & Y &= \gamma x_1 + \delta y_1, \quad w \in [\g, \g],
\end{alignat*}
where $\alpha,\beta,\gamma,\delta \in \C$.
Notice that $\Delta := \alpha\delta - \beta\gamma \neq 0$ because $\varphi$ induces an automorphism of $\g / [\g, \g]$.

\begin{lemma}\label{le:aut level4}
	Either $\beta=\gamma=0$ or $\alpha=\delta=0$.
\end{lemma}
\begin{proof}
	We have $[y_1,[x_1,[x_1,y_1]]]=[y_1,x_3]=0$. Applying $\varphi$ gives
	\[ [\varphi(y_1), [\varphi(x_1), [\varphi(x_1),\varphi(y_1)]]] = 0, \]
	so by the grading, we also have $[Y,[X,[X,Y]]]=0$. On the other hand, using
	$[y_1,x_3]=[x_1,y_3]=0$, $[x_1,x_3]=x_4$ and $[y_1,y_3]=y_4$, we compute
	\[
		[X,Y] = \Delta x_2,\qquad
		[X,[X,Y]] = \Delta(\alpha x_3 + \beta y_3),\qquad
		[Y,[X,[X,Y]]] = \Delta(\alpha\gamma\,x_4 + \beta\delta\,y_4).
	\]
	Hence $\alpha\gamma = \beta\delta = 0$.
	As $\Delta = \alpha\delta - \beta\gamma \neq 0$, the result follows.
\end{proof}

\begin{lemma}\label{le:aut level5}
	We have $\beta = \gamma = 0$ and $\alpha = \delta \neq 0$.
\end{lemma}
\begin{proof}
	We have
	\begin{multline*}
		[x_1,[x_1,[x_1,[x_1,y_1]]]] = [x_1,x_4] = x_5 + i y_5, \quad \text{with} \\
		x_5=[[x_1,y_1],[x_1,[x_1,y_1]]] \quad \text{and} \quad
			y_5=[[x_1,y_1],[y_1,[x_1,y_1]]].
	\end{multline*}
	Applying $\varphi$ and using the grading, we get
	\begin{equation}\label{eq:deg5}
		[X,[X,[X,[X,Y]]]] = [[X,Y],[X,[X,Y]]] + i\,[[X,Y],[Y,[X,Y]]].
	\end{equation}
	We now distinguish between the two different cases from \cref{le:aut level4}.
	
	Assume first that $\beta=\gamma=0$, so we have $X=\alpha x_1$ and $Y=\delta y_1$. Then each term of~\eqref{eq:deg5} is a scalar multiple of the
	corresponding bracket of $x_1$ and $y_1$, so \eqref{eq:deg5} becomes
	\[ \alpha^4\delta (x_5 + i y_5) = \alpha^3 \delta^2 x_5 + i \alpha^2\delta^3 y_5 . \]
	Comparing coefficients of $x_5$ gives $\alpha^4\delta=\alpha^3\delta^2$, i.e., $\alpha=\delta$ (because $\alpha \neq 0$ and $\delta \neq 0$).
	
	Assume next that $\alpha=\delta=0$, so we have $X=\beta y_1$ and $Y=\gamma x_1$. Then
	\begin{gather*}
		[X,Y] = -\beta\gamma x_2, \qquad
		[X,[X,Y]] = -\beta^2\gamma y_3,\qquad
		[Y,[X,Y]] = -\beta\gamma^2 x_3, \\
		[X,[X,[X,[X,Y]]]] = -\beta^4\gamma [y_1,y_4] = -\beta^4\gamma (i x_5 + y_5),
	\end{gather*}
	while the right-hand side of \eqref{eq:deg5} equals
	\[ \beta^3\gamma^2 [x_2,y_3] + i\beta^2\gamma^3[x_2,x_3] = \beta^3\gamma^2 y_5 + i \beta^2\gamma^3 x_5 . \]
	Comparing coefficients of $y_5$ and of $x_5$ gives
	\[ -\beta^4\gamma = \beta^3 \gamma^2 \quad \text{and} \quad -\beta^4\gamma = \beta^2 \gamma^3, \]
	i.e., $\beta=-\gamma$ and $\beta^2=-\gamma^2$ (because $\beta \neq 0$ and $\gamma \neq 0$), a contradiction.
\end{proof}

\begin{proposition}\label{pr:key}
Every automorphism $\varphi$ of $\g$ induces a scalar multiple of the identity on $\g/[\g,\g]$.
\end{proposition}
\begin{proof}
	This is simply a restatement of \cref{le:aut level5}.
\end{proof}

\section{The Lie algebra $\g$ has no real form}

We are now ready to prove our main result.

\begin{theorem}\label{thm:main}
	The $9$-dimensional nilpotent complex Lie algebra $\g$ from Definition~\ref{def} is isomorphic to its conjugate $\cg$, but $\g$ has no real form.
\end{theorem}
\begin{proof}
	Let $\tau \colon \g \to \cg$ be the isomorphism from \cref{pr:tau}.
	Suppose that $\g$ has a real form, and let $\psi \colon \g \to \cg$ be as in \cref{le:realform}.
	Then $\varphi := \tau^{-1} \circ \psi$ is an automorphism of $\g$.
	By \cref{pr:key}, $\varphi$ acts as scalar multiplication by $\lambda$ on $\g / [\g, \g]$ for some $\lambda \in \C^\times$.
	Then
	\[ \varphi(x_1) = \lambda x_1 + v \quad \text{and} \quad \varphi(y_1) = \lambda y_1 + w \]
	for some $v,w \in [\g, \g]$, hence
	\[ \psi(x_1) = \tau(\varphi(x_1)) = \lambda \tau(x_1) + \tau(v) = \lambda y'_1 + \tau(v) \quad \text{with } \tau(v) \in [\cg, \cg] . \]
	Also, by the definition of $\psi$, we have
	\[ \overline\psi(y_1') = \overline{\psi(y_1)} = \overline{\tau(\varphi(y_1))} = \overline{\lambda \tau(y_1) + \tau(w)}
		= \overline{- \lambda x'_1 + \tau(w)} = - \overline\lambda x_1 + \overline{\tau(w)} , \]
	with $\overline{\tau(w)} \in [\g, \g]$. However, by \cref{le:realform}, we have $\overline\psi \circ \psi = \id$, hence
	\[ x_1 = (\overline\psi \circ \psi)(x_1) = -\lambda \overline{\lambda} x_1 + z \quad \text{with } z \in [\g, \g] , \]
	a contradiction because $- \lambda \overline{\lambda} = - \lvert \lambda \rvert^2 \neq 1$ and $\g = \langle x_1, y_1 \rangle \oplus [\g, \g]$.
\end{proof}

\section{AI declaration}

We have been assisted by Claude Opus 5.5 by Anthropic for the construction and discussion of our example. We have not used AI to write down the results.

\bibliographystyle{alphaurl}
\bibliography{Lie9dim.bib}

\end{document}